\documentclass[a4paper,draft]{amsart}
\usepackage{amsrefs}
\usepackage[all]{xypic}
\SelectTips{cm}{}

\newcommand{\calA}{\mathcal{A}}
\newcommand{\calC}{\mathcal{C}}
\newcommand{\calD}{\mathcal{D}}
\newcommand{\calI}{\mathcal{I}}
\newcommand{\calR}{\mathcal{R}}
\newcommand{\calS}{\mathcal{S}}
\newcommand{\calT}{\mathcal{T}}
\newcommand{\calX}{\mathcal{X}}
\newcommand{\calY}{\mathcal{Y}}
\newcommand{\calZ}{\mathcal{Z}}

\newcommand{\bbA}{\mathbb{A}}
\newcommand{\bbN}{\mathbb{N}}
\newcommand{\bbZ}{\mathbb{Z}}

\let\Im=\undefined
\let\mod=\undefined
\DeclareMathOperator{\Ab}{Ab} %
\DeclareMathOperator{\tr}{tr} %
\DeclareMathOperator{\id}{id} %
\DeclareMathOperator{\Im}{Im} %
\DeclareMathOperator{\End}{End} %
\DeclareMathOperator{\ind}{ind} %
\DeclareMathOperator{\Hom}{Hom} %
\DeclareMathOperator{\mod}{mod} %
\DeclareMathOperator{\Mod}{Mod} %
\DeclareMathOperator{\rad}{rad} %
\DeclareMathOperator{\proj}{proj} %
\DeclareMathOperator{\KGdim}{KGdim} %
\DeclareMathOperator{\gldim}{gl.dim} %
\DeclareMathOperator{\umod}{\ul{mod}} %

\newcommand{\ul}{\underline}

\newcommand{\vertexD}[1]{\circ \save*+!D{\scriptstyle #1} \restore}
\newcommand{\vertexL}[1]{\circ \save*+!L{\scriptstyle #1} \restore}

\newcommand{\vertexU}[1]{\circ \save*+!U{\scriptstyle #1} \restore}

\newcounter{claim}[section]

\newtheorem{lemma}[claim]{Lemma}
\newtheorem{proposition}[claim]{Proposition}
\newtheorem{theorem}[claim]{Theorem}
\newtheorem{corollary}[claim]{Corollary}

\newtheorem*{maintheorem}{Main Theorem}

\theoremstyle{remark}
\newtheorem{remark}[claim]{Remark}

\author{Grzegorz Bobi\'nski}

\address{Faculty of Mathematics and Computer Science \\ Nicolaus
Copernicus University \\ ul.~Chopina 12/18 \\ 87-100 Toru\'n \\
Poland}

\email{gregbob@mat.umk.pl}

\author{Henning Krause}

\address{Faculty of Mathematics \\ Bielefeld University \\ D-33501
  Bielefeld \\ Germany}

\email{hkrause@mat.uni-bielefeld.de}

\thanks{The authors acknowledge the support from Collaborative
  Research Centre 701 \emph{Spectral Structures and Topological
    Methods in Mathematics}.  The first named author was also
  supported by National Science Center Grant No.\
  DEC-2011/03/B/ST1/00847.}

\title[The Krull--Gabriel dimension of discrete derived
  categories]{The Krull--Gabriel dimension of discrete derived
  categories}

\begin{document}

\begin{abstract}
  We compute the Krull--Gabriel dimension of the category of perfect
  complexes for finite dimensional algebras which are derived
  discrete.
\end{abstract}

\maketitle

\section*{Introduction}

Let $k$ be an algebraically closed field and $\Lambda$ a finite
dimensional $k$-algebra. We denote by $\mod\Lambda$ the category of
finitely presented $\Lambda$-modules and by $\proj\Lambda$ the full
subcategory of finitely generated projective $\Lambda$-modules.

The Krull--Gabriel dimension of the representation theory of $\Lambda$
is an invariant first studied by Geigle \cite{Geigle}. For this
invariant one considers the abelian category $\calC=\Ab(\mod\Lambda)$
of finitely presented functors $\mod\Lambda\to\Ab$ into the category
of abelian groups. The Krull--Gabriel dimension $\KGdim\calC$
of $\calC$ is by definition the smallest integer $n$ such that $\calC$
admits a filtration by Serre subcategories
\[0=\calC_{-1}\subseteq\calC_0 \subseteq\ldots
\subseteq\calC_n=\calC,\] where $\calC_i/\calC_{i-1}$ is the full
subcategory of all objects of finite length in $\calC/\calC_{i-1}$.

We have $\KGdim\calC=0$ if and only if $\Lambda$ is of finite
representation type by a classical result of Auslander
\cite{Auslander}, and $\KGdim\calC\neq1$ by a result of Herzog
\cite{Herzog} and Krause \cite{Kr1998}. In his thesis \cite{Geigle},
Geigle proved that $\KGdim\calC=2$, when $\Lambda$ is tame hereditary.

In this work we investigate the category of perfect complexes which is
by definition the bounded derived category $\calD^b(\proj\Lambda)$. We compute
the Krull--Gabriel dimension of the ablian category $\Ab (\calD^b
(\proj \Lambda))$,  when $\Lambda$ is derived discrete in the sense of
Vossieck \cite{Vossieck}. The main result is the following.

\begin{maintheorem} Let $\Lambda$ be a finite dimensional $k$-algebra.
\begin{enumerate}
\item If $\Lambda$ is derived discrete and piecewise hereditary, then
\[ \KGdim \Ab (\calD^b (\proj \Lambda)) = 0.\]
\item If $\Lambda$ is derived discrete and not piecewise hereditary,
then
\[ \KGdim \Ab (\calD^b (\proj \Lambda)) =
\begin{cases} 1 & \text{if $\gldim \Lambda = \infty$}, \\ 2 & \text{if
$\gldim \Lambda < \infty$}.
\end{cases}
\]
\item If $\Lambda$ is not  derived discrete, then
  \[ \KGdim \Ab (\calD^b (\proj \Lambda)) \geq 2.\]
\end{enumerate}
\end{maintheorem}

The rest of this note is devoted to proving this theorem. For an
elementary description of the Krull--Gabriel dimension, see
Proposition~\ref{pr:KG}.

\subsection*{Conventions}
By $\bbZ$, $\bbN$, and $\bbN_+$, we denote the sets of integers,
nonnegative integers, and positive integers, respectively. For $i, j
\in \bbZ$, set \[[i, j] := \{ l \in \bbZ \mid i \leq l \leq j \}.\]
Furthermore, $[i, \infty) := \{ l \in \bbZ \mid i \leq l \}$ and $(-
\infty, j] := \{ l \in \bbZ \mid l \leq j \}$.

\section{Derived discrete algebras}

Let $\Lambda$ be a finite dimensional $k$-algebra. The algebra
$\Lambda$ is called \emph{derived discrete} if for each sequence
$(h_n)_{n \in \bbZ}$ of nonnegative integers there are only finitely
many isomorphism classes of indecomposable objects $X$ in $\calD^b
(\proj \Lambda)$ such that $\dim_k H^n (X) = h_n$ for each $n \in
\bbZ$. Note that Vossieck's original definition \cite{Vossieck}
uses the category $\calD^b (\mod \Lambda)$, but he has shown that both
versions are equivalent. The one we use is more adequate in our setup.

In \cite{Vossieck}, it is shown that an algebra $\Lambda$ is derived
discrete if and only if either $\Lambda$ is piecewise hereditary of
Dynkin type or $\Lambda$ is a one-cycle gentle algebra not satisfying
the clock condition. Recall from \cite{Happel} that $\Lambda$ is
\emph{piecewise hereditary} if it is derived equivalent to a finite
dimensional hereditary algebra.  The class of one-cycle gentle algebras
not satisfying the clock condition has been further studied
in~\cite{BobinskiGeissSkowronski}. There it is shown that if $\Lambda$
is a derived discrete algebra and not piecewise hereditary of Dynkin
type, then $\Lambda$ is derived equivalent to an algebra of the form
$\Lambda (r, n, m)$, for some triple $(r, n, m) \in \Omega$.  Here,
$\Omega$ denotes the set of all triples $(r, n, m)$ of nonnegative
integers such that $1 \leq r \leq n$, and $\Lambda (r, n, m)$ is the
path algebra of the quiver
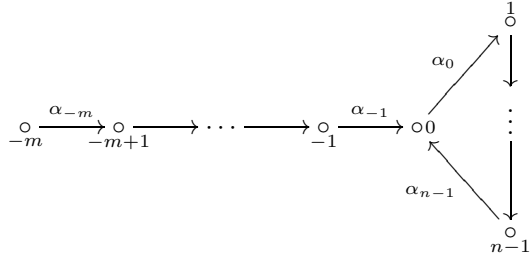
\[
\vcenter{\xymatrix{
& & & & & \vertexD{1} \ar[d]
\\
\vertexU{-m} \ar[r]^{\alpha_{-m}} & \vertexU{- m + 1} \ar[r] &
\cdots \ar[r] & \vertexU{-1} \ar[r]^{\alpha_{-1}} & \vertexL{0}
\ar[ru]^{\alpha_0} &  \vdots\ar[d]
\\
& & & & & \vertexU{n - 1} \ar[lu]^{\alpha_{n-1}}&}}
\]
bound by the relations
\[
\alpha_{n - r + 1} \alpha_{n - r}, \ldots, \alpha_{n - 1} \alpha_{n - 2}, \alpha_0 \alpha_{n - 1}.
\]

Prototypical examples to have in mind are the algebra $\Lambda(1,1,0)$
which equals the algebra $k[\varepsilon]$ of dual numbers
($\varepsilon^2=0$), and its Auslander algebra $\Lambda(1,2,0)$. Note
that $\gldim\Lambda(1,1,0)=\infty$ while $\gldim\Lambda(1,2,0)=2$.

\section{Krull--Gabriel dimension} \label{section dimension}

Let $\calC$ be an abelian category. A full subcategory
$\calC'\subseteq\calC$ is called a \emph{Serre subcategory} if it is
closed under subobjects, quotients and extensions. If
$\calC'\subseteq\calC$ is a Serre subcategory, then one defines the
\emph{quotient category} $\calC / \calC'$ as follows. The objects of
$\calC / \calC'$ coincide with the objects of $\calC$, and if $X$ and
$Y$ are objects of $\calC$, then
\[ \Hom_{\calC / \calC'} (X, Y) := \varinjlim \Hom_{\calC} (X', Y /
Y'),
\] where $X'$ and $Y'$ run through all subobjects of $X$ and $Y$,
respectively, such that $X / X'$ and $Y'$ belong to $\calC'$.

Following Gabriel \cite{Gabriel}*{IV.1} and Geigle \cite{Geigle}*{\S2},
the \emph{Krull--Gabriel dimension} $\KGdim \calC$ of $\calC$ is
defined as follows. Let $\calC_{-1} := 0$, and for each $n \in \bbN$
denote by $\calC_n$ the full subcategory of all objects $X$ in $\calC$
which are of finite length, when viewed as objects of $\calC /
\calC_{n - 1}$. Then $\KGdim \calC$ equals the smallest $n$ such that
$\calC_n = \calC$ (and $\infty$ when such $n$ does not exist).

Let $\calT$ be a triangulated category.  Following Freyd
\cite{Fr1966}*{\S3} and Verdier \cite{Ve1996}*{II.3}, we consider the
\emph{abelianisation} $\Ab(\calT)$ of $\calT$ which is the abelian
category of finitely presented functors $F\colon\calT\to\Ab$ into the
category $\Ab$ of abelian groups. Recall that a functor
$F\colon\calT\to\Ab$ is \emph{finitely presented} (\emph{finitely
  generated}, respectively) if there exists an exact sequence of the
form $H_Y \to H_X \to F \to 0$ ($H_X \to F \to 0$, respectively).
Here, for an object $X$ in $\calT$, we denote by $H_X$ the
representable functor $\Hom (X, -) \colon \calT \to \Ab$. Similarly,
if $f \colon X \to Y$ is a morphism in $\calT$, then we denote by
$H_f$ the induced morphism $H_Y \to H_X$.  The cohomological functor
$\iota \colon \calT \to \Ab (\calT)$ sending $X \in \calT$ to $H_X$ is
universal in the following sense. If $\varphi \colon \calT \to \calA$
is a contravariant cohomological functor, then there exists a unique
exact functor $\varphi' \colon \Ab (\calT) \to \calA$, such that
$\varphi = \varphi' \circ \iota$.

Now let $\Lambda$ be any ring. We wish to compute the Krull--Gabriel
dimension of $\Ab (\calD^b (\proj \Lambda))$ and begin with an
elementary observation. To this end fix a modular lattice $L$. Denote
by $L'$ the quotient which is obtained by collapsing all finite length
intervals in $L$. Set $L_{-1}=L$ and $L_n=(L_{n-1})'$ for
$n\in\bbN$. The \emph{dimension} of $L$ is the smallest $n$ such that $L=0$.

\begin{lemma}[{\cite{Kr1998}*{Lemma~1.1}}]\label{le:KG}
  Let $\calC$ be an abelian category and $X$ an object. For the
  lattice $L_{\calC}(X)$ of subobjects we have
   $L_{\calC}(X)_n \cong L_{\calC/\calC_n}(X)$ for all $n\in\bbN$.\qed
\end{lemma}

This lemma suggests an alternative description of the Krull--Gabriel
dimension which avoids the formation of quotient categories.

\begin{proposition}\label{pr:KG}
Let $\Lambda$ be any ring.
\begin{enumerate}
\item The finitely generated subfunctors of the forgetful functor
  $\mod \Lambda\to\Ab$ form a modular lattice and its dimension equals
  $\KGdim\Ab(\mod\Lambda)$.
\item The finitely generated subfunctors of $H^0\colon\calD^b (\proj
  \Lambda)\to\Ab$ form a modular lattice and its dimension equals
  $\KGdim\Ab(\calD^b (\proj \Lambda))$.
\end{enumerate}
\end{proposition}
\begin{proof}
  (1) The lattice of finitely generated subfunctors
  of $F=\Hom_\Lambda(\Lambda,-)$ equals the lattice of subobjects of
  $F$ in $\Ab(\mod\Lambda)$. Given a Serre subcategory $\calC\subseteq
  \Ab(\mod\Lambda)$, we have $F\in\calC$ iff $\calC=\Ab(\mod\Lambda)$.
  Now apply Lemma~\ref{le:KG}.

  (2) The lattice of finitely generated subfunctors of
  $H^0$ equals the lattice of subobjects of $H_\Lambda$ in
  $\Ab(\calD^b (\proj \Lambda))$, where $\Lambda$ is viewed as a
  complex concentrated in degree zero.  Now apply Lemma~\ref{le:KG},
  keeping in mind that $\Lambda$ generates $\calD^b (\proj \Lambda))$
  as a triangulated category.
\end{proof}

From now on suppose that $\Lambda$ is a finite dimensional $k$-algebra and set
$\calC:=\Ab (\calD^b (\proj \Lambda))$.  For the description of
$\calC_0$ one uses the well-known fact that the simple objects in
$\calC$ correspond to the Auslander--Reiten triangles in $\calD^b
(\proj \Lambda)$. Namely, if $X \xrightarrow{f} Y \to Z \to \Sigma X$
is an Auslander--Reiten triangle in $\calD^b (\proj \Lambda)$, then
$H_X / \Im H_f$ is a simple object in $\calC$, and every simple object
in $\calC$ is of this form. This follows directly from the definition
of an Auslander--Reiten triangle. Consequently, if $F \in \calC$, then
$F \in \calC_0$ if and only if \[\sum_{X \in \ind \calD^b (\proj
  \Lambda)} \dim_k F (X) < \infty,\] where $\ind \calD^b (\proj
\Lambda)$ denotes a fixed set of representatives of the indecomposable
objects in $\calD^b (\proj \Lambda)$; see also~\cite{Auslander}*{\S2}
for the above description of simple and finite length objects. This
condition immediately implies the first part of the Main Theorem.

\begin{proposition}
  Let $\Lambda$ be a derived discrete algebra which is piecewise
  hereditary. Then
\[ \KGdim \Ab (\calD^b (\proj \Lambda)) = 0.
\]
\end{proposition}

\begin{proof} If $\Lambda$ is a derived discrete algebra, which is
piecewise hereditary, then $\Lambda$ is piecewise hereditary of Dynkin
type. The well-known description of $\calD^b (\proj \Lambda)$ in this
case (see for example~\cite{Happel}), immediately implies that
\[\sum_{X \in \ind \calD^b (\proj \Lambda)} \dim_k H_M (X) < \infty\]
for each complex $M$ in $\calD^b (\proj \Lambda)$. Consequently,
\[\sum_{X \in \ind \calD^b (\proj \Lambda)} \dim_k F (X) < \infty\] for
each $F \in \Ab (\calD^b (\proj \Lambda))$, hence the claim follows.
\end{proof}

\section{Krull--Gabriel dimension and generic objects}
\label{section generic}

Let $\Lambda$ be a finite dimensional $k$-algebra. Generic modules
were introduced by Crawley-Boevey in order to describe the
representation type of an algebra \cites{CB1991,CB1992}. Following
\cite{GeKr2002}, an indecomposable object $X$ of the unbounded derived
category $\calD (\Mod \Lambda)$ of all $\Lambda$-modules is called
\emph{generic}, if $H^i (X)$ is a finite length $\End (X)$-module, for
each $i \in \bbZ$, but $X$ is not in $\calD^b(\mod \Lambda)$. Derived
discrete algebras can be characterised in terms of generic
complexes. This follows from work of Bautista~\cite{Bautista} and we
recall the following result.

\begin{proposition}[{\cite{Bautista}*{Theorem~1.1}}]
  Let $\Lambda$ be a finite dimensional $k$-algebra which is not
  derived discrete. Then there exists a generic object $X$ in $\calD
  (\Mod \Lambda)$ such that the division ring $\End (X) / \rad \End
  (X)$ contains an element which is transcendental over $k$.\qed
\end{proposition}

Note that the description of the endomorphism ring in
\cite{Bautista}*{Theorem~1.1} is a consequence of the proof which uses
\cite{CB1992}*{Theorem~9.5}.

  Next we combine Bautista's result with an argument due to Herzog
  \cite{Herzog}.  To be precise, Herzog proves a result about the
  abelianisation $\Ab(\mod\Lambda)$, but the same argument works for
$\Ab (\calD^b (\proj \Lambda))$ and yields the following.

\begin{proposition}[{\cite{Herzog}*{Theorem~3.6}}]
  \pushQED{\qed} Let $\Lambda$ be a finite dimensional $k$-algebra. If
  there exists a generic object $X$ in $\calD (\Mod \Lambda)$ such
  that $\End (X) / \rad \End (X)$ contains an element which is
  transcendental over $k$, then
\[ \KGdim \Ab (\calD^b (\proj \Lambda)) \geq 2.\qedhere\]
\end{proposition}

Our claim about the Krull--Gabriel dimension of an algebra which is
not derived discrete is an immediate consequence.

\begin{corollary}
\pushQED{\qed}  Let $\Lambda$ be a finite dimensional $k$-algebra which is not
  derived discrete. Then
  \[ \KGdim \Ab (\calD^b (\proj \Lambda)) \geq 2.\qedhere\]
\end{corollary}

\section{The category of perfect complexes}
\label{section_category}

Throughout this section we fix $(r, n, m) \in \Omega$ such that $r <
n$ and we put $\Lambda := \Lambda (r, n, m)$. Note that the condition
$r < n$ is equivalent to $\gldim \Lambda < \infty$. In this section we
follow~\cite{Bobinski} and describe a quiver $\Gamma$ together with a
set $\calR$ of relations such that the category $\calD^b (\proj
\Lambda)$ is equivalent to the path category $k\Gamma$ modulo the
given relations (see for example \cite{Ringel}*{\S2.1} for the
definition of $k\Gamma$). We refer to~\cite{BrPaPl} for a detailed
study of morphisms in $\calD^b (\proj \Lambda)$; the diagrams in there
might help to understand our calculations.

For $i \in [0, r - 1]$ we set
\begin{align*}
I_i &:= \bbZ^2,\\ I_i' &:= \{ (a, b) \in \bbZ^2 \mid a \leq b +
\delta_{i, 0} \cdot m \},\\
I_i'' &:= \{ (a, b) \in \bbZ^2 \mid a + \delta_{i, 0} \cdot n \leq b \},
\end{align*}
where $\delta_{x, y}$ is the Kronecker delta. The set of vertices
of $\Gamma$ is
\begin{multline*}
\Gamma_0:=\{X_v^{(i)}\mid i \in [0, r - 1],\,v \in
I_i'\}\,\cup\,\{Y_v^{(i)}\mid i \in [0, r - 1],\,v \in I_i''\}\\
\cup \,\{Z_v^{(i)}\mid i \in [0, r - 1],\,v \in I_i\}.
\end{multline*}

Now we describe the arrows in $\Gamma$ and associate to each
arrow a \emph{degree}. There are three cases.

(1) Fix $i \in [0, r - 1]$ and $v := (a, b) \in I_i'$. We put
\begin{align*}
  \calI_v'^{(i)} &:= [a, b + \delta_{i, 0} \cdot m ] \times [b,
  \infty), \\ \calX_v^{(i)} &:= [a, b + \delta_{i, 0}
  \cdot m] \times \bbZ, \\
  \calX_v'^{(i)} &:= (-\infty, a + \delta_{i, r - 1} \cdot m] \times
  [a, b + \delta_{i, 0} \cdot m].
\end{align*} For $u \in \calI_v'^{(i)}$, $u \neq v$, there is an
arrow $f_{v, u}'^{(i)} \colon X_v^{(i)} \to X_u^{(i)}$ of degree
$0$. Next, for $u \in \calX_v^{(i)}$ there is an arrow $g_{v,
u}'^{(i)} \colon X_v^{(i)} \to Z_u^{(i)}$ of degree $1$. Finally, for
$u \in \calX_v'^{(i)}$ there is an arrow $e_{v, u}'^{(i)} \colon
X_v^{(i)} \to X_u^{(i + 1)}$ of degree $2$, where we always change the
upper index modulo $r$.

(2) Fix $i \in [0, r - 1]$ and $v := (a, b) \in I_i''$. We put
\begin{align*} \calI_v''^{(i)} &:= [a, b - \delta_{i, 0} \cdot n]
\times [b, \infty), \\\calY_v^{(i)} &:= \bbZ \times [a, b -
\delta_{i, 0} \cdot n], \\ \calY_v'^{(i)} &:=
(-\infty, a - \delta_{i, r - 1} \cdot n] \times [a, b - \delta_{i, 0}
\cdot n].
\end{align*} For $u \in \calI_v''^{(i)}$, $u \neq v$, there is an
arrow $f_{v, u}''^{(i)} \colon Y_v^{(i)} \to Y_u^{(i)}$ of degree
$0$. Next, for $u \in \calY_v^{(i)}$ there is an arrow $g_{v,
u}''^{(i)} \colon Y_v^{(i)} \to Z_u^{(i)}$ of degree $1$. Finally, for
$u \in \calY_v'^{(i)}$ there is an arrow $e_{v, u}''^{(i)} \colon
Y_v^{(i)} \to Y_u^{(i + 1)}$ of degree $2$.

(3) Fix $i \in [0, r - 1]$ and $v := (a, b) \in I_i$. We put
\begin{align*} \calI_v^{(i)} &:= [a, \infty) \times [b, \infty), \\
\calZ_v'^{(i)} &:= (-\infty, a + \delta_{i, r - 1} \cdot m] \times [a,
\infty), \\ \calZ_v''^{(i)} &:= (-\infty, b - \delta_{i, r - 1} \cdot
n] \times [b, \infty), \\   \calZ_v^{(i)} &:= (-\infty,
a + \delta_{i, r - 1} \cdot m] \times (\infty, b - \delta_{i, r - 1}
\cdot n].
\end{align*} For $u \in \calI_v^{(i)}$, $u \neq v$, there is an arrow
$f_{v, u}^{(i)} \colon Z_v^{(i)} \to Z_u^{(i)}$ of degree $0$. Next,
for $u \in \calZ_v'^{(i)}$ there is an arrow $h_{v, u}'^{(i)} \colon
Z_v^{(i)} \to X_u^{(i + 1)}$ of degree $1$. Similarly, for $u \in
\calZ_v''^{(i)}$ there is an arrow $h_{v, u}''^{(i)} \colon Z_v^{(i)}
\to Y_u^{(i + 1)}$ of degree $1$. Finally, for $u \in \calZ_v^{(i)}$
there is an arrow $e_{v, u}^{(i)} \colon Z_v^{(i)} \to Z_u^{(i + 1)}$
of degree $2$.

Now we describe the set $\calR$ of relations. Let $f \colon X \to Y$
and $g \colon Y \to Z$ be arrows of degree $p$ and $q$,
respectively. If there is an arrow $h \colon X \to Z$ of degree $p +
q$, then we have the relation $g f = h$, otherwise we have the
relation $g f = 0$.

We summarise our construction.

\begin{proposition} \label{proposition category} There exists a
$k$-linear equivalence $k \Gamma / \calR\xrightarrow{\sim}\calD^b (\proj
\Lambda)$.
\end{proposition}

\begin{proof} The Auslander--Reiten quiver of $\calD^b (\proj
  \Lambda)$ has been described in~\cite{BobinskiGeissSkowronski}. It
  consists of $2 r$ components of type $\bbZ \bbA_\infty$ (they
  correspond to $X$- and $Y$-vertices of $\Gamma)$ and $r$ components
  of type $\bbZ \bbA_\infty^\infty$ (they correspond to $Z$-vertices
  of $\Gamma$). Moreover, under the action of the shift $\Sigma$ the
  components fall into three orbits, consisting of $r$ components
  each. The objects lying on the border of $\bbZ \bbA_\infty$
  components have been also identified. Using string combinatorics
  \cites{Crawley-Boevey, Krause} it is straightforward to verify the
  description of $H_X$ for $X$ lying on the border of a component of
  type $\bbZ \bbA_\infty$. By induction, using Auslander--Reiten
  triangles, the description of $H_X$ follows for each $X$ in the
  components of type $\bbZ \bbA_\infty$. Then we verify this
  description for two (cleverly) chosen neighbouring objects in a
  component of type $\bbZ \bbA_\infty^ \infty$ and proceed again by
  induction (using Auslander--Reiten triangles) to finish the proof.
\end{proof}

For  future use we introduce the following
notation. For $i \in [0, r - 1]$ set
\[f_{v, v}'^{(i)} := \id_{X_v^{(i)}}, \; v \in I_i',
\quad f_{v, v}''^{(i)} := \id_{Y_v^{(i)}}, \; v \in I_i'', \quad
 f_{v, v}^{(i)} := \id_{Z_v^{(i)}}, \; v \in I_i.
\]
Also, we let $f_{v, u}'^{(i)}$ denote the zero morphism $X_v^{(i)} \to
0$, if $v \in I_i'$, and $u \not \in I_i'$. The same convention applies
to $f_{v, u}''^{(i)}$, $h_{v, u}'^{(i)}$ and $h_{v, u}''^{(i)}$.

\section{The Krull--Gabriel dimension for algebras of finite global
  dimension} \label{section calculations}

Let $(r, n, m) \in \Omega$ be such that $r < n$. We consider $\calC :=
\Ab (\calD^b (\proj \Lambda))$ for $\Lambda := \Lambda (r, n, m)$. Our
aim is to prove that $\calC_1 \neq \calC$, but $\calC_2 = \calC$. In
order to show the latter claim, it is enough to prove that $H_U \in
\calC_2$ for each indecomposable $U \in \calD^b (\proj \Lambda)$. In
fact, we will prove that $H_U$ is either zero or simple in $\calC /
\calC_1$. Note that in order to prove that $H_U$ (more generally, $H_U
/ \Im H_g$, where $g \colon U \to M$ is a morphism) is either zero or
simple in $\calC / \calC_{n - 1}$ for some $n \in \bbN$, it is enough
to prove that, for every non-zero map $f \colon U \to V$ with $V$
indecomposable, either $\Im H_f \underset{\calC / \calC_{n - 1}}{=} 0$
or $\Im H_f \underset{\calC / \calC_{n - 1}}{=} H_U$ (either $\Im H_f
/ (\Im H_f \cap \Im H_g) \underset{\calC / \calC_{n - 1}}{=} 0$ or
$\Im H_f / (\Im H_f \cap \Im H_g) \underset{\calC / \calC_{n - 1}}{=}
H_U / \Im H_g$), where the subscript means that the equalities hold in
$\calC / \calC_{n - 1}$.

We begin with the description of the
simple objects in $\calC$.

\begin{lemma} \label{lemma simple0} The simple object in $\calC$ are
\begin{enumerate}

\item $A_v'^{(i)} := H_{X_v^{(i)}} / \Im H_{(f_{v, v + (1, 0)}'^{(i)},
f_{v, v + (0, 1)}'^{(i)})^{\tr}}$, $i \in [0, r - 1]$, $v \in I_i'$,

\item $A_v''^{(i)} := H_{Y_v^{(i)}} / \Im H_{(f_{v, v + (1,
0)}''^{(i)}, f_{v, v + (0, 1)}''^{(i)})^{\tr}}$, $i \in [0, r - 1]$,
$v \in I_i''$,

\item $A_v^{(i)} := H_{Z_v^{(i)}} / \Im H_{(f_{v, v + (1, 0)}^{(i)},
f_{v, v + (0, 1)}^{(i)})^{\tr}}$, $i \in [0, r - 1]$, $v \in I_i$.

\end{enumerate}
\end{lemma}

\begin{proof}
  This follows from the well-known description of the
  Auslander--Reiten triangles in $\calD^b(\proj\Lambda)$; see
  Section~\ref{section dimension}.
\end{proof}

Now we move to the category $\calC / \calC_0$. We first describe the
simple objects.

\begin{lemma} \label{lemma simple1} The objects
\begin{enumerate}

\item $B_{a, b, b'}'^{(i)} := H_{X_{(a, b)}^{(i)}} / \Im H_{(f_{(a,
b), (a + 1, b)}'^{(i)}, g_{(a, b), (a, b')}'^{(i)})^{\tr}}$ for $i \in
[0, r - 1]$, $(a, b) \in I_i'$, $b' \in \bbZ$,

\item $B_{a, b, b'}''^{(i)} := H_{Y_{(a, b)}^{(i)}} / \Im H_{(f_{(a,
b), (a + 1, b)}''^{(i)}, g_{(a, b), (b', a)}''^{(i)})^{\tr}}$ for $i
\in [0, r - 1]$, $(a, b) \in I_i''$, $b' \in \bbZ$,

\item $C_{a, b, b'}'^{(i)} := H_{Z_{(a, b)}^{(i)}} / \Im H_{(f_{(a,
b), (a + 1, b)}^{(i)}, h_{(a, b), (b', a)}'^{(i)})^{\tr}}$ for $i \in
[0, r - 1]$, $(a, b) \in I_i$, $b' \in (-\infty, a + \delta_{i, r - 1}
\cdot m + 1]$,

\item $C_{a, b, a'}''^{(i)} := H_{Z_{(a, b)}^{(i)}} / \Im H_{(f_{(a,
b), (a, b + 1)}^{(i)}, h_{(a, b), (a', b)}''^{(i)})^{\tr}}$ for $i \in
[0, r - 1]$, $(a, b) \in I_i$, $a' \in (-\infty, b -\delta_{i, r - 1}
\cdot n + 1]$,

\end{enumerate} are simple in $\calC / \calC_0$.
\end{lemma}

\begin{remark} One may easily show that every simple object in $\calC
/ \calC_0$ is (up to isomorphism) of the above form. It is also not
difficult to describe the isomorphism classes of the above objects.
\end{remark}

\begin{proof} We only prove the first claim; the remaining ones are
  proved similarly. Let $i \in [0, r - 1]$, $(a, b) \in I_i'$ and $b'
  \in \bbZ$. It is clear that $B_{a, b, b'}'^{(i)} \underset{\calC /
    \calC_0}{\neq} 0$. This follows, since for each $n \in \bbN$ we
  have in $\calC$ a short exact sequence
\[ 0 \to B_{a, b + n + 1, b'}'^{(i)} \to B_{a, b + n, b'}'^{(i)} \to
A_{(a, b + n)}'^{(i)} \to 0.
\]

For a non-zero morphism $f \colon X_{(a, b)}^{(i)} \to V$ with $V$
indecomposable we put
\[ B_f' := \Im H_f / (\Im H_f \cap \Im H_{(f_{(a, b), (a + 1,
b)}'^{(i)}, g_{(a, b), (a, b')}'^{(i)})^{\tr}}).
\] We have to show that either $B_f' \underset{\calC / \calC_0}{=} 0$
or $B_f' \underset{\calC / \calC_0} = B_{a, b, b'}'^{(i)}$ for every
$f$ as above. We may assume that we are in one of the following cases:
\begin{enumerate}

\item $V = X_{(c, d)}^{(i)}$ and $f = f_{(a, b), (c, d)}'^{(i)}$ for
some $(c, d) \in \calI_{(a, b)}'^{(i)}$,

\item $V = Z_{(c, d)}^{(i)}$ and $f = g_{(a, b), (c, d)}'^{(i)}$ for
some $(c, d) \in \calX_{(a, b)}^{(i)}$,

\item $V = X_{(c, d)}^{(i + 1)}$ and $f = e_{(a, b), (c, d)}'^{(i)}$
for some $(c, d) \in \calX_{(a, b)}'^{(i)}$.

\end{enumerate}

\textit{Case~$(1)$}. If $c > a$, then $f$ factors through $f_{(a, b),
  (a + 1, b)}'^{(i)}$, hence $B_f'$ is the zero subobject of $B_{a, b,
  b'}'^{(i)}$. If $c = a$, then we prove by induction on $d$ that
$B_f' \underset{\calC / \calC_0}{=} B_{a, b, b'}'^{(i)}$. Indeed, if
$d = b$, then the claim is obvious. If $d > b$, then we have a short
exact sequence
\[
0 \to B_f' \to B_{f_{(a, b), (a, d - 1)}'^{(i)}}' \to A_{(a, d - 1)}'^{(i)} \to 0,
\]
hence $B_f' \underset{\calC / \calC_0}{=} B_{f_{(a, b), (a, d -
    1)}'^{(i)}}'$ by Lemma~\ref{lemma simple0}. Moreover, $B_{f_{(a,
    b), (a, d - 1)}'^{(i)}}' \underset{\calC / \calC_0}{=} B_{a, b,
  b'}'$ by induction.

\textit{Case~$(2)$}. We prove that $B_f' \underset{\calC / \calC_0}{=}
0$ in this case. Again, we may assume that $c = a$. If $d \geq b'$,
then $f$ factors through $g_{(a, b), (a, b')}'^{(i)}$, hence $B_f'$ is
again zero. Finally, if $d < b'$, then we have a short exact sequence
\[ 0 \to B_{g_{(a, b), (c, d + 1)}'^{(i)}}' \to B_f' \to A_{(a, d)}
\to 0,
\] hence the claim follows by an obvious induction.

\textit{Case~$(3)$}. In this case $f$ factors through $(f_{(a, b), (a
+ 1, b)}'^{(i)}, g_{(a, b), (a, b')}'^{(i)})^{\tr}$, hence $B_f'$ is
zero.
\end{proof}

Now we show that some representable functors have finite length in
$\calC / \calC_0$.

\begin{lemma} \label{lemma finite1} Let $i \in [0, r - 1]$.
\begin{enumerate}

\item If $v \in I_i'$, then $H_{X_v^{(i)}} \in \calC_1$.

\item If $v \in I_i''$, then $H_{Y_v^{(i)}} \in \calC_1$.

\end{enumerate}
\end{lemma}

\begin{proof} Again, we only prove the first claim. Let $v = (a,
  b)$. The claim is shown by induction on $b - a$. If $b - a =
  \delta_{i, 0} \cdot m$, then we have a short exact sequence
\[ 0 \to C_{a, 0, a + \delta_{i, r - 1} \cdot m + 1}'^{(i)} \to
H_{X_v^{(i)}} \to B_{a, b, 0}'^{(i)} \to 0,
\] hence the claim follows from Lemma~\ref{lemma simple1}. If $b - a >
\delta_{i, 0} \cdot m$, then we have exact sequences
\begin{gather*} H_{X_{(a + 1, b)}^{(i)}} \to H_{X_v^{(i)}} \to
H_{X_v'^{(i)}} / \Im H_{f_{(a, b), (a + 1, b)}'^{(i)}} \to 0 \\
\intertext{and}  0 \to C_{a, 0, a + \delta_{i, r - 1} \cdot m +
1}'^{(i)} \to H_{X_v^{(i)}} / \Im H_{f_{(a, b), (a + 1, b)}'^{(i)}}
\to B_{a, b, 0}'^{(i)} \to 0,
\end{gather*} and the claim follows by induction and Lemma~\ref{lemma
simple1}.
\end{proof}

Next we show that the remaining representable functors corresponding
to the indecomposable objects in $\calD^b (\proj \Lambda)$ are not of
finite length in $\calC / \calC_0$.

\begin{lemma} \label{lemma nonsimple1} If $i \in [0, r - 1]$ and $v
\in I_i$, then $H_{Z_v^{(i)}} \not \in \calC_1$.
\end{lemma}

\begin{proof} Let $v = (a, b)$. For each $n \in \bbN$ we have the
following exact sequence
\[ 0 \to \Im H_{f_{(a, b), (a + n + 1, b)}^{(i)}} \to \Im H_{f_{(a,
b), (a + n, b)}^{(i)}} \to C_{a + n, b, a + \delta_{i, r - 1} \cdot m
+ 1}'^{(i)} \to 0,
\] which implies the claim (note that $\Im H_{f_{(a, b), (a,
b)}^{(i)}} = H_{Z_v^{(i)}}$).
\end{proof}

Finally, we show that $\calC_2 = \calC$. For this we only need to
prove the following.

\begin{lemma} If $i \in [0, r - 1]$ and $v \in I_i$, then
$H_{Z_v^{(i)}} \in \calC_2$.
\end{lemma}

\begin{proof} Let $v = (a, b)$. We know from Lemma~\ref{lemma
nonsimple1} that $H_{Z_v^{(i)}}$ is a non-zero object in
$\calC_1$. In order to prove it is simple we fix a non-zero morphism
$f \colon Z_v^{(i)} \to V$. We may assume that we are in one of the
following cases:
\begin{enumerate}

\item $V = Z_{(c, d)}^{(i)}$ and $f = f_{(a, b), (c, d)}^{(i)}$ for
some $(c, d) \in \calI_{(a, b)}^{(i)}$,

\item $V = X_{(c, d)}^{(i + 1)}$ and $f = h_{(a, b), (c, d)}'^{(i)}$
for some $(c, d) \in \calZ_{(a, b)}'^{(i)}$,

\item $V = Y_{(c, d)}^{(i + 1)}$ and $f = h_{(a, b), (c, d)}''^{(i)}$
for some $(c, d) \in \calZ_{(a, b)}''^{(i)}$,

\item $V = Z_{(c, d)}^{(i + 1)}$ and $f = e_{(a, b), (c, d)}^{(i)}$
for some $(c, d) \in \calZ_{(a, b)}^{(i)}$.

\end{enumerate}

\textit{Case~$(1)$}. We prove by induction on $c + d$ that $\Im H_f
\underset{\calC / \calC_1}{=} H_{Z_v^{(i)}}$. If $c + d = a + b$
(i.e., $c = a$ and $d = b$), the claim is obvious. Assume $c > a$.
Then we have an exact sequence
\[ 0 \to \Im H_f \to \Im H_{f_{(a, b), (c - 1, d)}^{(i)}} \to C_{c -
1, d, a + \delta_{i, r - 1} \cdot m + 1}'^{(i)} \to 0,
\] hence $\Im H_f \underset{\calC / \calC_1}{=} \Im H_{f_{(a, b), (c - 1,
d)}^{(i)}}$ by Lemma~\ref{lemma simple1}. Moreover, we have by induction
$\Im H_{f_{(a, b), (c - 1, d)}^{(i)}} \underset{\calC / \calC_1}{=}
H_{Z_v^{(i)}}$. We proceed similarly if $d > b$.

\textit{Case~$(2)$}. We have an epimorphism $H_{X_{(c, d)}^{(i + 1)}}
\to \Im H_f$, hence $\Im H_f \underset{\calC / \calC_1}{=} 0$ by
Lemma~\ref{lemma finite1}.

\textit{Case~$(3)$}. Analogous to Case~(2).

\textit{Case~$(4)$}. Again, we have an epimorphism $H_{X_{(c, a)}^{(i +
1)}} \to \Im H_f$, hence we get $\Im H_f \underset{\calC /
\calC_1}{=} 0$ (in fact one may even prove that $\Im H_f \in \calC_0$
in this case).
\end{proof}

\section{The algebras of infinite global dimension}

Throughout this section fix $n \in \bbN_+$ and $m \in \bbN$. We put
$\Lambda := \Lambda (n, n, m)$. The aim of this section is to prove
that $\KGdim \calC = 1$, where $\calC := \Ab (\calD^b (\proj
\Lambda))$. The basic idea is to extract the $X$-part of the arguments
in the finite global dimension case. We explain this in more detail.

First we describe the category $\calD^b (\proj \Lambda)$. Let $\Gamma$
be the quiver with the vertices $X_v^{(i)}$ for $i \in [0, n - 1]$ and
$v \in I_i$, where
\[ I_i := \{ (a, b) \in \bbZ^2 \mid a \leq b + \delta_{i, 0} \cdot m
\}.\] For $i \in [0, n - 1]$ and $v = (a, b) \in I_i$ we define
\begin{align*} \calI_v^{(i)} &:= [a, b + \delta_{i, 0} \cdot m ]
\times [b, \infty) \\ \calX_v^{(i)} &:= (-\infty, a +
\delta_{i, n - 1} \cdot m] \times [a, b + \delta_{i, 0} \cdot m].
\end{align*} Then for each $i \in [0, n - 1]$, $v \in I_i^{(i)}$, and
$u \in \calI_v^{(i)}$, $u \neq v$, we have an arrow $f_{v, u}^{(i)} :
X_v^{(i)} \to X_u^{(i)}$ of degree $0$, and for each $i \in [0, n -
1]$, $v \in I_i^{(i)}$, and $u \in \calX_v^{(i)}$, we have an arrow
$e_{v, u}^{(i)} : X_v^{(i)} \to X_u^{(i + 1)}$ of degree $1$. Finally,
by $\calR$ we denote the set of the following relations. Let $f \colon
X \to X'$ and $g \colon X' \to X''$ be arrows of degree $p$ and $q$,
respectively. If there is an arrow $h \colon X \to X''$ of degree $p +
q$, then we have the relation $g f = h$, otherwise we have the
relation $g f = 0$ (an explicit list of relations can be found
in~\cite{Bobinski}*{\S5}).

\begin{proposition} There exists a $k$-linear equivalence $k \Gamma /
  \calR\xrightarrow{\sim}\calD^b (\proj \Lambda)$.
\end{proposition}

\begin{proof} Analogous to the proof of Proposition~\ref{proposition
category}.
\end{proof}

It is obvious that $\calC_0 \neq \calC$. In order to prove $\calC_1 =
\calC$, it suffices to show that $H_U \in \calC_1$ for each
indecomposable $U \in \calD^b (\proj \Lambda)$. The arguments are
similar to those used in Section~\ref{section calculations} and we
state the analogues of Lemmas~\ref{lemma simple0}, \ref{lemma simple1}
and~\ref{lemma finite1} without proofs. Again, we use the convention that
$f_{v, u}^{(i)}$ ($e_{v, u}^{(i)}$) denotes the zero morphism
$X_v^{(i)} \to 0$ if $i \in [0, n - 1]$, $v \in I_i$, and $u \not \in
I_i$ ($u \not \in I_{i + 1}$, respectively).

\begin{lemma} The simple objects in $\calC$ are
\[ A_v^{(i)} := H_{X_v^{(i)}} / \Im H_{(f_{v, v + (1, 0)}^{(i)}, f_{v,
v + (0, 1)}^{(i)})^{\tr}}
\] for $i \in [0, n - 1]$ and $v \in I_i$.\qed
\end{lemma}

\begin{lemma} The objects
\[ B_{a, b, b'}^{(i)} := H_{X_{(a, b)}^{(i)}} / \Im H_{(f_{(a, b), (a
+ 1, b)}^{(i)}, e_{(a, b), (b', a)}^{(i)})^{\tr}}
\] for $i \in [0, n - 1]$, $(a, b) \in I_i$, and $b' \in (-\infty, a +
\delta_{i, n - 1} \cdot m]$, are simple in $\calC / \calC_0$.\qed
\end{lemma}

\begin{lemma} If $i \in [0, n - 1]$ and $v \in I_i$, then
$H_{X_v^{(i)}} \in \calC_1$.\qed
\end{lemma}

\section{Concluding remarks}

There are two other important triangulated categories, which one often
studies for a finite dimensional algebra $\Lambda$: the bounded
derived category $\calD^b (\mod \Lambda)$ and the stable category
$\umod \hat{\Lambda}$ of the repetitive algebra $\hat{\Lambda}$.  Thus
one may also ask about the Krull--Gabriel dimensions of the
abelianisations of these two categories. We have the following result.

\begin{theorem}
Let $\Lambda$ be a finite dimensional $k$-algebra and denote
by $\calC$ either $\Ab (\calD^b (\mod \Lambda))$ or  $\Ab (\umod
\hat{\Lambda})$. Then $\KGdim\calC\neq 1$, and $\KGdim\calC=0$ if and
only if $\Lambda$ is piecewise hereditary of Dynkin type.
\end{theorem}

\begin{proof}
  We apply the Main Theorem and use the chain of fully faithful exact
  functors
\[ \calD^b (\proj \Lambda)\to \calD^b (\mod \Lambda)\to \umod \hat{\Lambda}.\]

 If $\Lambda$ is piecewise hereditary of
  Dynkin type, then $\Lambda$ is derived discrete and $\calD^b (\proj
  \Lambda) = \calD^b (\mod \Lambda) = \umod \hat{\Lambda}$. Thus
  $\KGdim\calC=0$.

  If $\Lambda$ is not derived discrete, then Lemma~\ref{lemma inequality} below yields
  \[ 2 \leq \KGdim \Ab (\calD^b (\proj \Lambda)) \leq \KGdim \Ab
  (\calD^b (\mod \Lambda)) \leq \KGdim \Ab (\umod \hat{\Lambda}).\]

Finally, assume $\Lambda$ is derived discrete, but not piecewise
hereditary. If $\gldim \Lambda < \infty$, then again $\calD^b (\proj
\Lambda) = \calD^b (\mod \Lambda) = \umod \hat{\Lambda}$. Thus assume
$\gldim \Lambda = \infty$. In this case the description of $\umod
\hat{\Lambda}$ is the same as the description of $\calD^b (\proj
\Lambda)$ given in Section~\ref{section_category}, hence the arguments
from Section~\ref{section calculations} apply. On the other hand,
$\calD^b (\mod \Lambda)$ lies strictly between $\calD^b (\proj
\Lambda)$ and $\umod \hat{\Lambda}$, but some of the $Z$-modules from
Section~\ref{section_category} survive
(see~\cite{BobinskiGeissSkowronski} for details) and the corresponding
representable functors are not of finite length in $\calC / \calC_0$
when $\calC = \Ab (\calD^b (\umod \hat{\Lambda}))$.
\end{proof}

In the above proof the following lemma is used.

\begin{lemma} \label{lemma inequality}
Let $\calS$ be a thick subcategory of a triangulated category $\calT$. Then $\KGdim \Ab (\calS) \leq \KGdim \Ab (\calT)$.
\end{lemma}

\begin{proof}
  The universal property of the abelianisation yields an exact
  embedding $\Ab (\calS) \to \Ab (\calT)$. An easy
  induction shows that $\Ab (\calT)_n \cap \Ab (\calS) \subseteq \Ab
  (\calS)_n$ for each $n \in \bbN$.
\end{proof}

We observe that
\[ \KGdim \Ab (\calD^b (\mod \Lambda)) = \KGdim \Ab (\umod
\hat{\Lambda})
\] if $\Lambda$ is derived discrete. We do not know whether this
equality holds for an arbitray finite dimensional $k$-algebra
$\Lambda$.

\bibsection

\end{document}